\documentclass[a4paper,10pt,english,french]{smfart}
\usepackage[english,frenchb]{babel}
\usepackage[T1]{fontenc}
\usepackage{lmodern}
\usepackage{graphicx}
\usepackage{amsmath,amsthm,amssymb,amsfonts}
\usepackage[a4paper,left=4cm,right=4cm,top=4cm,bottom=4cm]{geometry}
\numberwithin{equation}{section}
\title[Power-free values of the polynomial $t_1 \cdots t_r - 1$]{Power-free values of the polynomial $t_1 \cdots t_r - 1$}
\author{Pierre Le Boudec}
\subjclass{$11$N$32$}
\address{Université Denis Diderot (Paris VII) \\ Institut de Mathématiques de Jussieu \\ UMR 7586 \\ Case $7012$ - Bâtiment Chevaleret \\ Bureau $7$C$14$ \\ $75205$ Paris Cedex 13, France}
\email{pleboude@math.jussieu.fr}

\begin{document}

\makeatletter
\def\imod#1{\allowbreak\mkern10mu({\operator@font mod}\,\,#1)}
\makeatother

\newtheorem{lemma}{Lemma}
\newtheorem{theorem}{Theorem}
\newtheorem{corollaire}{Corollary}
\newtheorem{proposition}{Proposition}

\newcommand{\vol}{\operatorname{vol}}
\newcommand{\D}{\mathrm{d}}
\newcommand{\rank}{\operatorname{rank}}
\newcommand{\Pic}{\operatorname{Pic}}
\newcommand{\Gal}{\operatorname{Gal}}
\newcommand{\meas}{\operatorname{meas}}
\newcommand{\Spec}{\operatorname{Spec}}

\begin{abstract}
Let $k,r \geq 2$ be two integers. We prove an asymptotic formula for the number of $k$-free values of the $r$ variables polynomial $t_1 \cdots t_r - 1$ over $[1,x]^r \cap \mathbb{Z}^r$.
\end{abstract}

\maketitle

\section{Introduction}

Let $k,r \geq 2$ be two integers. Given a multivariable polynomial $P(t_1, \dots, t_r)$, a natural problem is to investigate the number of its $k$-free values over $[1,x]^r \cap \mathbb{Z}^r$. This problem being very hard, we can content ourselves with asking about the asymptotic behaviour of this number as $x$ grows to infinity.

Many authors have studied this problem for polynomials in few variables (see \cite{MR2773215} for a recent overview) but there are no general results for multivariable polynomials. However, using a result of Granville \cite{MR1654759} for polynomials in one variable, Poonen \cite{MR1980998} has proved under the $abc$ conjecture that the number of squarefree values of any multivariable polynomial $P(t_1, \dots, t_r)$ over the set $[1,x]^r \cap \mathbb{Z}^r$ divided by $x^r$ converges to a product of local densities. This proves conditionally that the number of squarefree values of any polynomial behaves quite nicely.

The purpose of this short paper is to attack this problem for the polynomial $t_1 \cdots t_r - 1$ and to take advantage of its particular shape to investigate deeper the asymptotic behaviour of the number $\mathcal{N}_{k,r}(x)$ of its $k$-free values over
$[1,x]^r \cap \mathbb{Z}^r$.

In our investigation, a trichotomy appears depending on the value of $r$. For $r = 2$, we will make use of Weil's bound for Kloosterman sums and for $r = 3$, we will use the work of Heath-Brown about the equidistribution of the values of the divisor function
$\tau_3 := 1 \ast 1 \ast 1$ in arithmetic progressions \cite{MR866901} (improving the earlier result of Friedlander and Iwaniec \cite{MR786351}). Finally, for $r \geq 4$, the most efficient way to tackle the problem is to use a result of Shparlinski \cite[Theorem $9$]{MR2303621}.

Let $\gamma_2 = 4/3$, $\gamma_3 = 8/5$, $\gamma_4 = 2$, $\gamma_5 = 40/19$ and, for $r \geq 6$,
\begin{eqnarray*}
\gamma_r & = & 3 \left( 1 - \frac{3}{r+5} \right) \textrm{,}
\end{eqnarray*}
and let
\begin{eqnarray*}
\delta_{k,r} & = & \left( 1 - \frac1{k} \right) \gamma_r \textrm{.}
\end{eqnarray*}
Our main result is the following.

\begin{theorem}
\label{Main}
Let $\varepsilon > 0$ be fixed. As $x$ tends to $+ \infty$, if $\delta_{k,r} \leq 1$ we have the estimate
\begin{eqnarray*}
\mathcal{N}_{k,r}(x) & = & c_{k,r} x^r + O \left( x^{r - \delta_{k,r} + \varepsilon} \right) \textrm{,}
\end{eqnarray*}
where
\begin{eqnarray*}
c_{k,r} & = & \prod_p \left( 1 - \frac1{p^k} \left( 1 - \frac1{p} \right)^{r-1} \right) \textrm{,}
\end{eqnarray*}
and if $1 < \delta_{k,r} \leq 2$ we have the estimate
\begin{eqnarray*}
\mathcal{N}_{k,r}(x) & = & c_{k,r} x^r - \theta_{k,r}^{(1)}(x) x^{r-1} + O \left( x^{r - \delta_{k,r} + \varepsilon} \right) \textrm{,}
\end{eqnarray*}
where
\begin{eqnarray*}
\theta_{k,r}^{(1)}(x) & = & r \sum_{d = 1}^{+ \infty} \frac{\mu(d)}{\varphi(d^k)} \left( \frac{\varphi(d)}{d} \right)^{r-1}
\sum_{m|d} \mu(m) \left\{ \frac{x}{m} \right\} \textrm{,}
\end{eqnarray*}
and, finally, if $\delta_{k,r} > 2$ we have the estimate
\begin{eqnarray*}
\mathcal{N}_{k,r}(x) & = & c_{k,r} x^r - \theta_{k,r}^{(1)}(x) x^{r-1} + \theta_{k,r}^{(2)}(x) x^{r-2}
+ O \left( x^{r - \delta_{k,r} + \varepsilon} \right) \textrm{,}
\end{eqnarray*}
where
\begin{eqnarray*}
\theta_{k,r}^{(2)}(x) & = &  \frac{r(r-1)}{2} \sum_{d = 1}^{+ \infty} \frac{\mu(d)}{\varphi(d^k)} \left( \frac{\varphi(d)}{d} \right)^{r-2} \left( \sum_{m|d} \mu(m) \left\{ \frac{x}{m} \right\} \right)^2 \textrm{.}
\end{eqnarray*}
\end{theorem}

The interest of theorem \ref{Main} lies more in the quality of the error term coming from the strength of the various results used rather than in the main term which is no surprise. Indeed, the constant $c_{k,r}$ has the following interpretation. For $q \geq 1$, let us denote by $\rho_r(q)$ the number of solutions to the equation $t_1 \cdots t_r - 1 = 0$ in
$\left( \mathbb{Z} / q \mathbb{Z} \right) ^r$, namely
\begin{eqnarray*}
\rho_r(q) & = & \# \{1 \leq t_1, \dots, t_r \leq n, \ t_1 \cdots t_r - 1 \equiv 0 \imod{q} \} \textrm{.}
\end{eqnarray*}
Clearly, $\rho_r(q) = \varphi(q)^{r-1}$ and we thus have
\begin{eqnarray*}
c_{k,r} & = & \prod_p \left( 1 - \frac{\rho_r(p^k)}{p^{kr}} \right) \textrm{.}
\end{eqnarray*}
Therefore, $c_{k,r}$ is actually a product of local densities.

It is worth pointing out that $\delta_{k,r} \leq 1$ if and only if $r = 2$ and $k = 2, 3, 4$ or $r = 3, 4$ and $k = 2$. Note that for the hardest case, namely $(k,r) = (2,2)$, we obtain
\begin{eqnarray*}
\mathcal{N}_{2,2}(x) & = &  c_{2,2} x^2 + O \left( x^{4/3 + \varepsilon} \right) \textrm{,}
\end{eqnarray*}
which is the result forecast by Tolev in \cite[Section $2$]{Tolev} where he proves a completely similar result, that is to say with the same error term, for squarefree values of the polynomial $t_1^2 + t_2^2 +1$.

In addition, it is straightforward to check that $\theta_{k,r}^{(1)}$ and $\theta_{k,r}^{(2)}$ are bounded functions. Furthermore, a short calculation yields
\begin{eqnarray*}
\theta_{k,r}^{(1)}(x) & = & r \sum_{m = 1}^{+ \infty} \left( \frac{|\mu(m)|}{\varphi(m^k)} \left( \frac{\varphi(m)}{m} \right)^{r-1}
\prod_{p \nmid m} \left( 1 - \frac1{p^k} \left( 1 - \frac1{p} \right)^{r-2} \right) \right)
\left\{ \frac{x}{m} \right\} \textrm{,}
\end{eqnarray*}
and thus $\theta_{k,r}^{(1)}$ is seen to be positive and, for instance, we have the bound
\begin{eqnarray*}
\theta_{k,r}^{(1)}(x) & \leq & r \frac{\zeta(k)}{\zeta(2k)} \textrm{.}
\end{eqnarray*}
It is a quite interesting fact to notice that the term $- \theta_{k,r}^{(1)}(x) x^{r-1}$ is thus a correcting term whose presence can be explained by the fact that if $n$ is a positive integer then the function $\mathcal{N}_{k,r}(x)$ is constant over the range
$[n, n+1[$. However, this does not mean that this term vanishes if we consider only integral values of $x$.

The following section is dedicated to the investigation of a quantity which will naturally appear in the main term of $\mathcal{N}_{k,r}(x)$ in the proof of theorem \ref{Main} and the last section is devoted to the proof of the theorem properly.

Along the proof, $\varepsilon$ is an arbitrary small positive number and, as a convention, the implicit constants involved in the notations $O$ and $\ll$ are allowed to depend on $k$, $r$ and $\varepsilon$. In addition, $\varphi$ denotes Euler's totient function, $\mu$ the Möbius function and $\{ \ \}$ and $\lfloor \ \rfloor$ respectively the fractional part and the floor part functions.

It is a great pleasure for the author to thank Professor de la Bretèche and Professor Browning for their careful reading of earlier versions of the manuscript and for their useful advice. The author is also extremely grateful to Professor Shparlinski for drawing his attention to the fact that making use of sums of multiplicative characters rather than Kloosterman sums is more efficient for
$r \geq 4$. This change of strategy has yielded a significant improvement of the main result in this case.

Part of this work was done while the author was attending the Initial Instructional Workshop of the semester GANT organized by the Centre Interfacultaire Bernoulli at the \'Ecole Polytechnique Fédérale de Lausanne. The hospitality and the financial support of this institution are gratefully acknowledged.

\section{Preliminary lemmas}

\label{Preliminary section}

Let $r \geq 2$. For $a,q \geq 1$ two coprime integers and $x \geq 1$, we introduce the quantity
\begin{eqnarray*}
\mathcal{S}_r(x;q,a) & = & \sum_{\substack{1 \leq t_1, \dots, t_r \leq x \\ t_1 \cdots t_r \equiv a \imod{q}}} 1 \textrm{.}
\end{eqnarray*}
This section is devoted to giving several different estimates for $\mathcal{S}_r(x;q,a)$ which will be used in the proof of theorem \ref{Main} in the following section. These estimates consist in proving that for fixed $q \geq 1$ and for varying $1 \leq a \leq q-1$ coprime to $q$, the quantities $\mathcal{S}_r(x;q,a)$ have a similar asymptotic behaviour. In other words, these estimates are equidistribution results.

The first of these estimates is proved by making use of Kloosterman sums and will actually be used only in the case $r = 2$. The second estimate deals with the case $r = 3$ and uses the work of Heath-Brown \cite{MR866901}. Finally, the third estimate is a result of Shparlinski and is only concerned with the case $r \geq 4$.

Even though we make use of the following estimate only in the case $r = 2$, we prove it for any $r \geq 2$ since this does not require much more effort.

\begin{lemma}
\label{preliminary}
Let $r \geq 2$ and $\varepsilon > 0$ be fixed. For $a, q \geq 1$ two coprime integers and $x \geq 1$, we have the estimate
\begin{eqnarray*}
\mathcal{S}_r(x;q,a) & = & \frac1{\varphi(q)} \sum_{\substack{1 \leq t_1, \dots, t_r \leq x \\ \gcd(t_1 \cdots t_r,q) = 1}} 1 + O \left( q^{(r-1)/2 + \varepsilon} \right) \textrm{.}
\end{eqnarray*}
\end{lemma}

\begin{proof}
We set $e_q(t) = e^{2 i \pi t /q}$. We reduce the variables $t_1, \dots, t_r$ to their residue classes modulo $q$ and we detect the congruences using sums of exponentials. We obtain
\begin{eqnarray}
\notag
\mathcal{S}_r(x;q,a) & = & \sum_{1 \leq t_1, \dots, t_r \leq x}
\sum_{\substack{\alpha_1, \dots, \alpha_r = 1 \\ \alpha_1 \cdots \alpha_r \equiv a \imod{q} }}^q
\prod_{i=1}^r \frac1{q} \sum_{\ell_i = 1}^q e_q(\ell_i(\alpha_i - t_i)) \\
\label{S}
& = & \frac1{q^r} \sum_{\ell_1, \dots, \ell_r = 1}^q K(\ell_1, \dots, \ell_r a, q) F_q(x;\ell_1, \dots, \ell_r) \textrm{,}
\end{eqnarray}
where $K(\ell_1, \dots, \ell_r a, q)$ is the $(r-1)$-dimensional Kloosterman sum given by
\begin{eqnarray*}
K(\ell_1, \dots, \ell_r a, q) & = & \sum_{\substack{\alpha_1, \dots, \alpha_r = 1 \\ \alpha_1 \cdots \alpha_r \equiv a \imod{q}}}^q
e_q ( \ell_1 \alpha_1 + \cdots + \ell_r \alpha_r ) \\
& = & \sum_{\substack{\alpha_1, \dots, \alpha_{r-1} = 1 \\ \gcd(\alpha_1 \cdots \alpha_{r-1},q) = 1}}^q
e_q \left( \ell_1 \alpha_1 + \cdots + \ell_{r-1} \alpha_{r-1} + \ell_r a \alpha_1^{-1} \cdots \alpha_{r-1}^{-1} \right) \textrm{,}
\end{eqnarray*}
where $\alpha^{-1}$ denotes the inverse of $\alpha$ modulo $q$ and where $F_q(x;\ell_1, \dots, \ell_r)$ is defined by
\begin{eqnarray*}
F_q(x;\ell_1, \dots, \ell_r) & = & \prod_{i=1}^r \ \sum_{1 \leq t_i \leq x} e_q( - \ell_i t_i ) \textrm{.}
\end{eqnarray*}
We use Weinstein's version of the works of Weil \cite{MR0027151} and Deligne \cite{MR0340258} (see
\cite[Theorems $1$ and $2$]{MR630958} and note that Smith has obtained similar results in \cite{MR544261}), namely
\begin{eqnarray*}
\left| K(\ell_1, \dots, \ell_r a, q) \right| & \leq & t_q r^{\omega(q)} q^{(r-1)/2} \prod_{j=1}^{r-1} \gcd(\ell_j,\ell_r,q)^{1/2} \textrm{,}
\end{eqnarray*}
where $t_q = 1$ if $q$ is odd and $t_q = 2^{(r + 1)/2}$ if $q$ is even and where $\omega(q)$ denotes the number of prime factors of $q$. Therefore, writing $t_q r^{\omega(q)} \ll q^{\varepsilon}$ where, as explained in the introduction, the constant involved is allowed to depend on $r$ and $\varepsilon$, and noticing that $\gcd(\ell_j,\ell_r,q)^{1/2} \leq \gcd(\ell_j,q)^{1/2}$, we get
\begin{eqnarray}
\label{K}
K(\ell_1, \dots, \ell_r a, q) & \ll & q^{(r-1)/2 + \varepsilon} \prod_{j=1}^{r-1} \gcd(\ell_j,q)^{1/2} \textrm{.}
\end{eqnarray}
We denote by $||x||$ the distance from $x$ to the set of integers. Note that if $\ell_i \neq q$ for all $i = 1, \dots, r$ then
$F_q(x;\ell_1, \dots, \ell_r)$ is a product of $r$ geometric sums and thus we have the bound
\begin{eqnarray}
\label{F}
F_q(x;\ell_1, \dots, \ell_r) & \ll & \prod_{i=1}^r \left| \left| \frac{\ell_i}{q} \right| \right|^{-1} \textrm{.}
\end{eqnarray}
Let $\mathcal{S}_r^{\ast}(x;q)$ be the sum of the terms of the expression \eqref{S} for which at least one of the $\ell_i$ is equal to $q$. This quantity is easily seen to be independent of $a$. The bound \eqref{K} for $K(\ell_1, \dots, \ell_r a, q)$ together with the bound \eqref{F} for $F_q(x;\ell_1, \dots, \ell_r)$ prove that
\begin{eqnarray*}
\mathcal{S}_r(x;q,a) - \mathcal{S}_r^{\ast}(x;q) & \ll & \frac1{q^r} q^{(r-1)/2 + \varepsilon}
\sum_{\ell_1, \dots, \ell_r = 1}^{q-1} \ \prod_{j=1}^{r-1} \gcd(\ell_j,q)^{1/2}
\prod_{i=1}^r \left| \left| \frac{\ell_i}{q} \right| \right|^{-1} \\
& \ll & \frac{q^{\varepsilon}}{q^{(r+1)/2}} \sum_{0 < |\ell_1|, \dots, |\ell_r | \leq q/2} \
\prod_{j=1}^{r-1} \gcd(\ell_j,q)^{1/2} \prod_{i=1}^r \frac{q}{\ell_i} \\
& \ll & q^{(r-1)/2 + \varepsilon} \sum_{d_1, \dots, d_{r-1} | q} \ 
\sum_{\substack{0 < |\ell_1 |, \dots, |\ell_r | \leq q/2 \\ d_1 | \ell_1, \dots , d_{r-1} | \ell_{r-1}}} \ \prod_{j=1}^{r-1} d_j^{1/2} \prod_{i=1}^r \frac1{\ell_i} \\
& \ll & q^{(r-1)/2 + \varepsilon} \log(q)^r \sum_{d_1, \dots, d_{r-1} | q} \ \prod_{j=1}^{r-1} d_j^{-1/2}  \textrm{,}
\end{eqnarray*}
and thus, after rescaling $\varepsilon$, we get
\begin{eqnarray}
\label{estimate}
\mathcal{S}_r(x;q,a) - \mathcal{S}_r^{\ast}(x;q) & \ll & q^{(r-1)/2 + \varepsilon} \textrm{.}
\end{eqnarray}
Recall that $\mathcal{S}_r^{\ast}(x;q)$ is independent of $a$. Averaging the estimate \eqref{estimate} over $a$ coprime to $q$ therefore proves that
\begin{eqnarray*}
\frac1{\varphi(q)} \sum_{\substack{1 \leq t_1, \dots, t_r \leq x \\ \gcd(t_1 \cdots t_r,q) = 1}} 1 - \mathcal{S}_r^{\ast}(x;q) & \ll & q^{(r-1)/2 + \varepsilon} \textrm{,}
\end{eqnarray*}
which completes the proof.
\end{proof}

We now state the result which will be used in the case $r = 3$.

\begin{lemma}
\label{preliminary 3}
Let $\varepsilon > 0$ be fixed. For $a, q \geq 1$ two coprime integers and $x \geq 1$ such that $q \leq x^{12/7}$, we have the estimate
\begin{eqnarray*}
\mathcal{S}_3(x;q,a) & = & \frac1{\varphi(q)} \sum_{\substack{1 \leq t_1, t_2, t_3 \leq x \\ \gcd(t_1 t_2 t_3,q) = 1}} 1
+ O \left( q^{1/4} x^{1 + \varepsilon} \right) \textrm{.}
\end{eqnarray*}
\end{lemma}

\begin{proof}
Let $0 < \delta \leq 1$ and $U_1$, $U_2$ and $U_3$ be variables running over the set
$\left\{ (1 + \delta)^n , n \in \mathbb{Z}_{\geq -1} \right\}$. Let us introduce the quantity
\begin{eqnarray*}
S(U_1,U_2,U_3;q,a) & = &  \#
\left\{ (t_1, t_2, t_3) \in \mathbb{Z}^3,
\begin{array}{l}
U_i < t_i \leq (1 + \delta) U_i, i \in \{1, 2, 3 \} \\
t_1 t_2 t_3 \equiv a \imod{q}
\end{array}
\right\} \textrm{.}
\end{eqnarray*}
By \cite[Lemma $5$]{MR866901}, there exists a quantity $M(U_1,U_2,U_3;q)$ independent of $a$ such that, for $U_1,U_2,U_3 \leq x$,
\begin{eqnarray}
\label{equidistribution}
\ \ \ \ \ S(U_1,U_2,U_3;q,a) - M(U_1,U_2,U_3;q) & \ll & x^{\varepsilon} \left( q^{5/6} + q^{1/4} x + q^{1/2} x^{1/2} \right) \textrm{.}
\end{eqnarray}
Let $f = \lfloor \log(x) / \log(2) \rfloor$ and let us choose $\delta = x^{1/(f+1)} - 1$. We have the equality
\begin{eqnarray*}
\mathcal{S}_3(x;q,a) & = & \sum_{U_1,U_2,U_3 \leq x} S(U_1,U_2,U_3;q,a) \textrm{.}
\end{eqnarray*}
The estimate \eqref{equidistribution} therefore shows that there exists a quantity $\mathcal{S}_3^{\ast}(x;q)$ independent of $a$ such that
\begin{eqnarray*}
\mathcal{S}_3(x;q,a) - \mathcal{S}_3^{\ast}(x;q) & \ll & x^{\varepsilon} \left( q^{5/6} + q^{1/4} x + q^{1/2} x^{1/2} \right) \textrm{.}
\end{eqnarray*}
For $q \leq x^{12/7}$, the error term $q^{1/4} x^{1 + \varepsilon}$ dominates and since $\mathcal{S}_3^{\ast}(x;q)$ is independent of~$a$, averaging this estimate over $a$ coprime to $q$ immediately concludes the proof.
\end{proof}

For $r \geq 4$, we will use another estimate for $\mathcal{S}_r(x;q,a)$. This estimate is due to Shparlinski
\cite[Theorem $9$]{MR2303621} and essentially draws upon Burgess bounds for sums of multiplicative characters (see
\cite[Theorem $2$]{MR0148626} and \cite[Theorem A]{MR838632}). Note that in \cite[Theorem $9$]{MR2303621}, the result is stated with the condition $x \leq q$ but it is easy to see that it remains true without this restriction.

\begin{lemma}
\label{preliminary'}
Let $r \geq 4$, $s \in \{ 2, 3 \}$ and $\varepsilon > 0$ be fixed. For $a, q \geq 1$ two coprime integers and $x \geq 1$, we have the estimate
\begin{eqnarray*}
\mathcal{S}_r(x;q,a) & = & \frac1{\varphi(q)} \sum_{\substack{1 \leq t_1, \dots, t_r \leq x \\ \gcd(t_1 \cdots t_r,q) = 1}} 1 + O \left( x^{\alpha_{r,s}} q^{\beta_{r,s} + \varepsilon} \right) \textrm{,}
\end{eqnarray*}
where $\alpha_{r,s}$ and $\beta_{r,s}$ are given by
\begin{eqnarray*}
\alpha_{r,s} & = & r - \frac{r - 4 + 2 s}{s} \textrm{,} \\
\end{eqnarray*}
and
\begin{eqnarray*}
\beta_{r,s} & = & \frac{(r-4)(s+1)}{4s^2} \textrm{.}
\end{eqnarray*}
\end{lemma}

\section{Proof of theorem \ref{Main}}

Let $k \geq 2$ be an integer. The Möbius function of order $k$ is defined by setting $\mu_k(1) = 1$ and, for $p$ a prime number and $\ell$ a positive integer,
\begin{eqnarray*}
\mu_k \left( p^{\ell} \right) & = &
\begin{cases}
1 & \textrm{ if } \ell \leq k - 2 \textrm{,} \\
-1 & \textrm{ if } \ell = k - 1 \textrm{,} \\
0 & \textrm{ otherwise,}
\end{cases}
\end{eqnarray*}
with the value of $\mu_k$ at any integer defined by multiplicativity. Note that $\mu_2$ is the usual Möbius function $\mu$. By construction, $|\mu_k|$ is the characteristic function of the set of $k$-free integers. The following elementary identity (see \cite[Lemma 5]{MR0253999}) is the starting point of our proof. For any integer
$n \geq 1$,
\begin{eqnarray*}
|\mu_k(n)| & = & \sum_{d^k|n} \mu(d) \textrm{.}
\end{eqnarray*}
We thus obtain
\begin{eqnarray*}
\mathcal{N}_{k,r}(x) & = & \sum_{\substack{1 \leq t_1, \dots, t_r \leq x \\ t_1 \cdots t_r \neq 1}} |\mu_k(t_1 \cdots t_r - 1)| \\
& = & \sum_{\substack{1 \leq t_1, \dots, t_r \leq x \\ t_1 \cdots t_r \neq 1}} \ \sum_{d^k | t_1 \cdots t_r - 1} \mu(d) \\
& = & \sum_{1 \leq d < x^{r/k}} \mu(d)
\sum_{\substack{1 \leq t_1, \dots, t_r \leq x \\ t_1 \cdots t_r \equiv 1 \imod{d^k} \\ t_1 \cdots t_r \neq 1}} 1 \textrm{.}
\end{eqnarray*}
Let $1 \leq y < x^{r/k}$ be a parameter to be specified later. We write $\mathcal{N}_{k,r}'(x)$ and $\mathcal{N}_{k,r}''(x)$ respectively for the contributions coming from the sums over $d$ for $1 \leq d \leq y$ and
$y < d < x^{r/k}$. It turns out that the contribution of $\mathcal{N}_{k,r}''(x)$ is negligible and we therefore start by proving an upper bound for $\mathcal{N}_{k,r}''(x)$. Denoting by $\tau_r$ the Dirichlet convolution of the constant arithmetic function equal to $1$ by itself $r$ times and using the elementary bound $\tau_r(n) \ll n^{\varepsilon}$, we easily obtain
\begin{eqnarray*}
\mathcal{N}''_{k,r}(x) & = & \sum_{y < d < x^{r/k}} \mu(d)
\sum_{\substack{1 \leq t_1, \dots, t_r \leq x \\ t_1 \cdots t_r \equiv 1 \imod{d^k} \\ t_1 \cdots t_r \neq 1}} 1 \\
& \leq & \sum_{y < d < x^{r/k}} \sum_{\substack{1 \leq n \leq x^r \\ n \equiv 1 \imod{d^k}}} \tau_r(n) \\
& \ll & x^{\varepsilon} \sum_{y < d < x^{r/k}} \sum_{\substack{1 \leq n \leq x^r \\ n \equiv 1 \imod{d^k}}} 1 \textrm{.}
\end{eqnarray*}
Since $d < x^{r/k}$, the inner sum is bounded by $2 x^r/d^k$. This yields
\begin{eqnarray}
\notag
\mathcal{N}''_{k,r}(x) & \ll & x^{r + \varepsilon} \sum_{y < d < x^{r/k}} \frac1{d^k} \\
\label{N''}
& \ll & \frac{x^{r + \varepsilon}}{y^{k-1}} \textrm{.}
\end{eqnarray}
We now turn to the estimation of $\mathcal{N}_{k,r}'(x)$. Using lemma \ref{preliminary}, we get
\begin{eqnarray*}
\mathcal{N}'_{k,2}(x) & = & \sum_{1 \leq d \leq y} \frac{\mu(d)}{\varphi(d^k)}
\sum_{\substack{1 \leq t_1, t_2 \leq x \\ \gcd(t_1 t_2, d) = 1}} 1
+ O \left( \sum_{1 \leq d \leq y} d^{k/2 + \varepsilon} \right) \\
& = & \sum_{1 \leq d \leq y} \frac{\mu(d)}{\varphi(d^k)}
\sum_{\substack{1 \leq t_1, t_2 \leq x \\ \gcd(t_1 t_2, d) = 1}} 1
+ O \left( x^{\varepsilon} y^{k/2+1} \right) \textrm{.}
\end{eqnarray*}
Using lemma \ref{preliminary 3}, we get, for $y \leq x^{12/7k}$,
\begin{eqnarray*}
\mathcal{N}'_{k,3}(x) & = & \sum_{1 \leq d \leq y} \frac{\mu(d)}{\varphi(d^k)}
\sum_{\substack{1 \leq t_1, t_2, t_3 \leq x \\ \gcd(t_1 t_2 t_3, d) = 1}} 1
+ O \left( x^{1 + \varepsilon} \sum_{1 \leq d \leq y} d^{k/4} \right) \\
& = & \sum_{1 \leq d \leq y} \frac{\mu(d)}{\varphi(d^k)}
\sum_{\substack{1 \leq t_1, t_2, t_3 \leq x \\ \gcd(t_1 t_2 t_3, d) = 1}} 1
+ O \left( x^{1 + \varepsilon} y^{k/4 + 1} \right) \textrm{.}
\end{eqnarray*}
Finally, using lemma \ref{preliminary'}, we get, for any $r \geq 4$ and $s \in \{2, 3 \}$, 
\begin{eqnarray*}
\mathcal{N}'_{k,r}(x) & = & \sum_{1 \leq d \leq y} \frac{\mu(d)}{\varphi(d^k)}
\sum_{\substack{1 \leq t_1, \dots, t_r \leq x \\ \gcd(t_1 \cdots t_r, d) = 1}} 1
+ O \left(  x^{\alpha_{r,s}} \sum_{1 \leq d \leq y} d^{k \beta_{r,s} + \varepsilon} \right) \\
& = & \sum_{1 \leq d \leq y} \frac{\mu(d)}{\varphi(d^k)}
\sum_{\substack{1 \leq t_1, \dots, t_r \leq x \\ \gcd(t_1 \cdots t_r, d) = 1}} 1
+ O \left( x^{\alpha_{r,s} + \varepsilon} y^{k \beta_{r,s} + 1} \right) \textrm{.}
\end{eqnarray*}
In addition, a Möbius inversion yields
\begin{eqnarray*}
\sum_{\substack{1 \leq t \leq x \\ \gcd(t,d) = 1}} 1 & = &  \sum_{m|d} \mu(m) \left\lfloor \frac{x}{m} \right\rfloor \\
& = & \frac{\varphi(d)}{d} x - \sum_{m|d} \mu(m) \left\{ \frac{x}{m} \right\} \textrm{.}
\end{eqnarray*}
This equality plainly gives
\begin{eqnarray*}
\sum_{\substack{1 \leq t_1, \dots, t_r \leq x \\ \gcd(t_1 \cdots t_r,d) = 1}} 1 & = &
\left( \frac{\varphi(d)}{d} \right)^r x^r
- r \left( \frac{\varphi(d)}{d} \right)^{r-1} \left( \sum_{m|d} \mu(m) \left\{ \frac{x}{m} \right\} \right) x^{r-1} \\
& & + \frac{r(r-1)}{2} \left( \frac{\varphi(d)}{d} \right)^{r-2} \left( \sum_{m|d} \mu(m) \left\{ \frac{x}{m} \right\} \right)^2 x^{r-2} + O \left( d^{\varepsilon} x^{r-3} \right) \textrm{.}
\end{eqnarray*}
Furthermore,
\begin{eqnarray*}
x^r \sum_{1 \leq d \leq y} \frac{\mu(d)}{\varphi(d^k)} \left( \frac{\varphi(d)}{d} \right)^r & = &
x^r \sum_{d = 1}^{+ \infty} \frac{\mu(d)}{\varphi(d^k)} \left( \frac{\varphi(d)}{d} \right)^r
+ O \left( \frac{x^{r}}{y^{k-1}} \right) \\
& = & c_{k,r} x^r + O \left( \frac{x^{r}}{y^{k-1}} \right) \textrm{.}
\end{eqnarray*}
Doing the same thing for the second and the third terms and setting
\begin{eqnarray*}
\mathcal{M}_{k,r}(x) & = & c_{k,r} x^r - \theta_{k,r}^{(1)}(x) x^{r-1} + \theta_{k,r}^{(2)}(x) x^{r-2} \textrm{,}
\end{eqnarray*}
we easily find that 
\begin{eqnarray*}
\sum_{1 \leq d \leq y} \frac{\mu(d)}{\varphi(d^k)}
\sum_{\substack{1 \leq t_1, \dots, t_r \leq x \\ \gcd(t_1 \cdots t_r, d) = 1}} 1 - \mathcal{M}_{k,r}(x) & \ll &
x^{r-3} + \frac{x^{r}}{y^{k-1}} \textrm{.}
\end{eqnarray*}
Our investigation has led us to the conclusion that
\begin{eqnarray*}
\mathcal{N}_{k,2}'(x) - \mathcal{M}_{k,2}(x) & \ll & x^{\varepsilon} y^{k/2+1} + \frac{x^2}{y^{k-1}} \textrm{,}
\end{eqnarray*}
and, for $y \leq x^{12/7k}$,
\begin{eqnarray*}
\mathcal{N}_{k,3}'(x) - \mathcal{M}_{k,3}(x) & \ll & x^{1 + \varepsilon} y^{k/4 + 1} + \frac{x^3}{y^{k-1}} \textrm{,}
\end{eqnarray*}
and, for any $r \geq 4$ and $s \in \{ 2, 3 \}$,
\begin{eqnarray*}
\mathcal{N}_{k,r}'(x) - \mathcal{M}_{k,r}(x) & \ll &
x^{\alpha_{r,s} + \varepsilon} y^{k \beta_{r,s} + 1} +  x^{r-3} + \frac{x^{r}}{y^{k-1}} \textrm{.}
\end{eqnarray*}
Recalling that $\mathcal{N}_{k,r}(x) = \mathcal{N}'_{k,r}(x) + \mathcal{N}''_{k,r}(x)$ and the bound \eqref{N''} for $\mathcal{N}''_{k,r}(x)$, we finally get
\begin{eqnarray*}
\mathcal{N}_{k,2}(x) - \mathcal{M}_{k,2}(x) & \ll & x^{\varepsilon} y^{k/2+1} + \frac{x^{2 + \varepsilon}}{y^{k-1}} \textrm{,}
\end{eqnarray*}
and, for $y \leq x^{12/7k}$,
\begin{eqnarray*}
\mathcal{N}_{k,3}(x) - \mathcal{M}_{k,3}(x) & \ll & x^{1 + \varepsilon} y^{k/4 + 1} + \frac{x^{3 + \varepsilon}}{y^{k-1}} \textrm{,}
\end{eqnarray*}
and, for any $r \geq 4$ and $s \in \{ 2, 3 \}$,
\begin{eqnarray*}
\mathcal{N}_{k,r}(x) - \mathcal{M}_{k,r}(x) & \ll &
x^{\alpha_{r,s} + \varepsilon} y^{k \beta_{r,s} + 1} + x^{r-3} + \frac{x^{r + \varepsilon}}{y^{k-1}} \textrm{.}
\end{eqnarray*}
We can now choose $y$ to our best advantage. We instantly see that for $r = 2$ the optimal value is $y = x^{4/3k}$ and for $r = 3$ the best choice is $y = x^{8/5k}$, which satisfies $y \leq x^{12/7k}$. Furthermore, for $r \geq 4$ the optimal value is
$y = x^{8r/k(3r+4)}$ if $s = 2$ and $y = x^{3(r+2)/k(r+5)}$ if $s = 3$. It is easy to see that if $r = 4$ or $r = 6$ then the two choices for $s$ yield the same result, for $r = 5$ it is better to choose $s = 2$ and for $r > 6$ it is better to choose $s = 3$. Finally, we can immediately check that in each case we obtain the result claimed.

\bibliographystyle{is-alpha}
\bibliography{biblio}

\end{document}